\documentclass[12pt]{amsart}

\usepackage{amsmath,amssymb,amsthm}
\usepackage{url}

\date{}
\author{Georg Braun}

\newtheorem{theorem}{Theorem}
\newtheorem{proposition}{Proposition}
\newtheorem{lemma}{Lemma}

\theoremstyle{definition}

\begin{document}

\title[Supercritical Branching Processes with Emigration]{On Supercritical Branching Processes with Emigration}

\thanks{{\em Address:}
Mathematisches Institut,
Universit\"at T\"ubingen,
Auf der Morgenstelle 10,
72076 T\"ubingen, Germany.
Email: georg.braun@uni-tuebingen.de}

\begin{abstract}
\noindent We study supercritical branching processes under the influence of an i.i.d.\ emigration component. We provide conditions, under which the lifetime of the process is finite respectively has a finite expectation. A new version of the Kesten-Stigum theorem is obtained and the extinction probability for a large initial population size is related to the tail behaviour of the emigration.
\end{abstract}

\maketitle

\section{Introduction}

Branching processes are a fascinating class of stochastic processes, which model the evolution of a population under the assumption that different individuals give independently of each other birth to a random number of children. In the present article, we will study the consequences of including an i.i.d.\ emigration component between consecutive generations in the supercritical regime. Intuitively speaking, the behaviour of this model is determined by the interplay of two opposite effects, namely the explosive nature of the branching processes and the decrease in the population size caused by emigration.

Formally, let $((\xi_{n,j})_{j \geq 1}, Y_n)_{n \geq 1}$ denote a sequence of i.i.d.\ random variables. Assume that the sequence $(\xi_{1,j})_{j \geq 1}$ is i.i.d.\ and that both $\xi:=\xi_{1,1}$ and $Y:=Y_1$ only take values in $\mathbb{N}_0$. Then we define a branching process with emigration $(Z_n)_{n \geq 0}$ by setting $Z_0 := k \in \mathbb{N}$ and recursively
\begin{align}
Z_{n+1} := \left( \sum\limits_{j=1}^{Z_n} \xi_{n+1,j} - Y_{n+1} \right)_+, \quad \quad n \geq 0 .
\end{align}
Throughout this article, we will focus on the supercritical case and more precisely assume that
\[ \lambda := E[ \xi] \in (1,\infty). \]
Naturally, our study will concentrate on the extinction time $\tau$ of $(Z_n)_{n \geq 0}$, which is defined by 
\[
\tau := \inf \{ n \geq 1 ~\vert~ Z_n = 0 \}, \quad \textnormal{where}~ \inf \emptyset := \infty.
\]
Amongst other things, we will prove that $\tau$ is almost surely finite if and only if $E[ \log_+ Y] = \infty$. Moreover, we will show that $E[ \tau] < \infty$ if there exists $\varepsilon > 0$ with
\begin{align}
\sum\limits_{n \geq 1} \prod\limits_{m=1}^n P \left[ Y \leq r (\lambda + \varepsilon)^m \right] < \infty \quad \textnormal{for~all~} r \in (0,\infty).
\end{align}
On the contrary, under some additional assumptions, we will also establish that $E[\tau]= \infty$ if there are $\theta \in (1,\infty)$ and $r \in (0,\infty)$ satisfying
\begin{align}
\sum\limits_{n \geq 1} \prod\limits_{m=1}^n P \left[ Y \leq r \lambda^m m^{-\theta} \right] = \infty.
\end{align}

The precise statements of all of our results are given in Section 2. We also present a new version of the Kesten-Stigum theorem and relate the behaviour the extinction probabilities
\[ q_k := P[\tau < \infty~\vert~Z_0=k], \quad k \geq 1, \]
to the tail behaviour of $Y$ as $k \rightarrow \infty$.\\

Our study is motivated by a simple observation, which links our model to subcritical autoregressive processes. We will explain the details in Section 3. To the best of our knowledge, this connection between branching processes with emigration and autoregressive processes has not been investigated in the literature so far.\\

The proofs of our results will be carried out in Section 4, 5 and 6.\\

While our criteria ensuring $E[ \tau] < \infty$ respectively $E[ \tau] = \infty$ are not exact, the following natural example illustrates that the gap in our characterisation is quite narrow.\\

\noindent \textit{Example 1}. Let $\lambda \in (1,\infty)$ and assume the existence of $c \in (0,\infty)$ and $n_0 \in \mathbb{N}$ with
\[  P \left[ \log Y > n \right] = \frac{c}{n} \quad \textnormal{for~all~} n \geq n_0.   \]
Then $E[ \log_+ Y] = \infty$ and hence $\tau < \infty$ almost surely. Moreover, as we will verify in Appendix 1, condition (2) holds if $c > \log \lambda$ and condition (3) is satisfied if $c \leq \log \lambda$.\\

\noindent Let us end this introduction by briefly mentioning previous literature results on branching processes with emigration.\\

The study of the critical case $\lambda = 1$ was initiated by Vatutin, who considered the case $Y \equiv 1$ and $\sigma^2 := \textnormal{Var}[\xi] \in (0,\infty)$ in \cite{Vat78}. Vatutin showed that $P[\tau > n]$ is regularly varying for $n \rightarrow \infty$ with exponent $-1-2/\sigma^2$, and, assuming that all moments of $\xi$ are finite, proved that $2 Z_n/n \sigma^2$, conditioned on being positive, converges weakly to the exponential distribution with mean one. These results were improved by Vinokurov and Kaverin in \cite{Vin87} and \cite{Kav91}, and more recently by Denisov, Korshunov and Wachtel in \cite{DKW16} for $\sigma^2 < 2$. The approach used in \cite{DKW16} relies on Markov chains with asymptotically zero drift and more generally allows a size-dependent offspring distribution as well as possibly both immigration and emigration.

More or less specific models of critical branching processes involving both immigration and emigration were studied by Nagaev and Khan in \cite{NK81} and by Yanev and Yanev in \cite{YY95}.

To the best of our knowledge, subcritical and supercritical branching processes with emigration have been previously studied mostly in continuous time. In this case, the population size changes if exactly one individual gives birth to a random number of children or if an emigration event, sometimes called catastrophe, occurs. If each individual has either $0$ or $2$ children, then the branching process, in fact, reduces to a birth-and-death process. The case of a catastrophe rate proportional to the population size was studied by Pakes in \cite{Pak86} and \cite{Pak89}. Moreover, Pakes studied this model with a size-independent emigration rate in \cite{PakI}, \cite{PakII} and \cite{PakIII}. In the supercritical regime, he related the almost sure extinction of the branching process to the condition $E[ \log_+ Y ] = \infty$, compare Theorem 2.1 and Corollary 3.2 in \cite{PakI}. This was also verified by Grey in \cite{Grey88}, who also proved this for our time-discrete model if $(\xi_{1,j})_{j \geq 1}$ and $Y_1$ are independent. As we have already stated, this independence assumption can be avoided.

\section{Preliminaries and Statement of Results}

To avoid degenerated cases, let us introduce a hypothesis (H), which we impose for the rest of this article. We denote by $(\tilde{Z}_n)_{n \geq 0}$ the renewal branching process given by $\tilde{Z}_0 := Z_0 = k$ and

\[
\displaystyle \tilde{Z}_{n + 1}  := \left\{ \begin{array}{c l}
\left( \displaystyle \sum\limits_{j=1}^{\tilde{Z}_n}  \xi_{n+1,j} - Y_{n+1} \right)_+ &\quad \textnormal{if}~\tilde{Z}_n \geq 1, \\ \\
k&\quad \textnormal{if}~\tilde{Z}_n = 0.
\end{array} \right. \]~\\
Note that by construction the Markov chains $(Z_n)_{n \geq 0}$ and $(\tilde{Z}_n)_{n \geq 0}$ share the same underlying state space, which however may depend on the choice of $Z_0 = \tilde{Z}_0 = k$.\\

\noindent In the following we shall not only assume $\lambda \in (1,\infty)$ but also\\

(H) The chain $\big( \tilde{Z}_n \big)_{n \geq 0}$ is irreducible and has an infinite state space.\\ 

\noindent In some of our results, we will also need Grey's restriction\\

(IND) The random variables $(\xi_{1,j})_{j \geq 1}$ and $Y_1$ are independent.\\
 
\noindent However, we have generally tried to avoid (IND) in our results. Let us mention that (H) holds if and only if the following two statements do.\\

(H1) $\quad \displaystyle P \bigg[ \sum\limits_{j=1}^k \xi_{1,k}  - Y_1 \geq k+1 \bigg] > 0$.\\

(H2) $\quad \displaystyle P \bigg[ \sum\limits_{j=1}^n \xi_{1,j} - Y_1 \leq n-1 \bigg] > 0~$ for all $n \geq 1$.\\

\noindent Roughly speaking, (H1) and (H2) ensure that neither emigration nor branching dominate each other completely for the possibly rather small initial state respectively for big population size. If $\lambda > 1$, then, by applying the law of large numbers, (H1) is always satisfied if $Z_0 = \tilde{Z}_0 = k \in \mathbb{N}$ is chosen large enough. On the other hand, (H2) holds, for example, if $P[ \xi =0]>0$ or if $Y$ is unbounded and (IND).\\

It is worth mentioning that both (H1) and (H2) are preserved if the number of initial individuals $Z_0 = k \geq 1$ is increased. The underlying state space of $(Z_n)_{n \geq 0}$ respectively $(\tilde{Z}_n)_{n \geq 0}$ may be affected by such a modification. However, this will not cause any problem in our study.\\

Let us now state the results.

\begin{theorem}
The following statements are equivalent.\\~\\
\begin{tabular}{c l}
\textnormal{(i)}&The process $\big( \tilde{Z}_n \big)_{n \geq 0}$ is recurrent, i.e.\ $\tau < \infty$ almost surely.\\ \\
\textnormal{(ii)}&$E[ \log_+ Y ] = \infty$.\\
\end{tabular}~\\
\end{theorem}

If the processes $(Z_n)_{n \geq 0}$ dies out almost surely, it is natural to ask whether its expected lifetime is finite or infinite. We provide the following answer to this question.

\begin{theorem}
Let $E[ \log_+ Y ] = \infty$.\\~\\
\begin{tabular}{c l}
\textnormal{(I)}&Assume that there are $\varepsilon > 0$ and $r \in (0,\infty)$ with\\ \\ 
&$\displaystyle \quad 0 < \sum\limits_{n \geq 1} \prod\limits_{m=1}^n P \left[ Y \leq r (\lambda + \varepsilon)^m \right]  < \infty$.\\ \\
&Then $E[ \tau] < \infty$, i.e.\ $(\tilde{Z}_n)_{n \geq 0}$ is positive recurrent.\\ \\
\textnormal{(II)}&Assume $E \big[ \xi^{1+\delta} \big] < \infty$ for a $\delta > 0$, \textnormal{(IND)}, and\\ \\
&$\displaystyle \quad \sum\limits_{n \geq 1} \prod\limits_{m=1}^n P \left[ Y \leq r \lambda^m m^{-\theta} \right]  = \infty \quad \textnormal{for~a~} \theta  \in (1, \infty),~ r \in (0,\infty)$.\\ \\
&Then $\tau < \infty$ a.s. and $E[ \tau] = \infty$, i.e.\ $\big( \tilde{Z}_n \big)_{n \geq 0}$ is nullrecurrent.\\
\end{tabular}
\end{theorem}

If the process $(Z_n)_{n \geq 0}$ survives forever with a positive probability, one might try to understand the distribution of $\tau$ and the extinction probabilities $(q_k)_{k \geq 1}$ in case of a large initial population size $k \geq 1$. For this purpose, we will use the concept of slow and regular variation in the sense of Karamata and assume\\

(REG) $P[Y>t]$ varies regularly for $t \rightarrow \infty$ with index $\alpha \in (0,\infty)$.\\

\noindent For a gentle introduction to slow and regular variation we refer the reader to \cite{Mik99}. A measurable function $L: [0,\infty) \rightarrow (0,\infty)$ is called slowly varying for $t \rightarrow \infty$, if for all $c \in (0, \infty)$ one has $L(ct)/L(t) \rightarrow 1$ as $t \rightarrow \infty$. Moreover, a measurable function $f: [0,\infty) \rightarrow (0,\infty)$ is regularly varying for $t \rightarrow \infty$, if there exists $\alpha \in \mathbb{R}$, $t_0 \in [0,\infty)$ and a slowly varying function $L$ satisfying $f(t) = t^\alpha L(t)$ for all $t \geq t_0$. In this case the constant $\alpha \in \mathbb{R}$ is unique and $- \alpha$ is called the index of $f$. 

\begin{theorem}
Assume \textnormal{(REG)} and let $N \in \mathbb{Z}_{\geq 2} \cup \{ \infty \}$. Then,
\begin{align*}
\limsup\limits_{k \rightarrow \infty} P[ \tau < N~\vert~Z_0 = k] ~P[Y>k]^{-1} \leq \sum\limits_{l=1}^{N-1} \lambda^{-\alpha l}.
\end{align*}
Furthermore, if all exponential moments of $\xi$ are finite, then
\[
\lim_{k \rightarrow \infty} ~P[ \tau < N~\vert~Z_0 = k]~P[Y>k]^{-1} = \sum\limits_{l=1}^{N-1} \lambda^{-\alpha l}.
\]
\end{theorem}

By choosing $N= \infty$ in Theorem 3 we in particular obtain results on the extinction probabilities $(q_k)_{k \geq 1}$ for $k \rightarrow \infty$.\\

Besides studying $\tau$ and $(q_k)_{k \geq 1}$, one can also try to understand the asymptotic behaviour of the process $(Z_n)_{n \geq 0}$ conditioned on its non-extinction. As in the case without any migration, Doob's martingale convergence theorem yields the existence of the almost sure limit
\[
W:= \lim\limits_{n \rightarrow \infty} \lambda^{-n} Z_n,
\]
which satisfies $0 \leq E[W] \leq k$.\\

\begin{theorem}~\\~\\
\noindent \begin{tabular}{l l}
\textnormal{(a)}&$P[W > 0] > 0$ if and only if\\ \\
\multicolumn{2}{c}{$\quad \quad \quad E[ \xi  \log_+ \xi] < \infty \quad \quad \textnormal{and} \quad \quad  E[ \log_+ Y] < \infty$.}\\ \\
&Furthermore, in this case\\ \\
\multicolumn{2}{c}{$\quad \quad \quad P[W > 0]=P[ \tau = \infty] $.}\\ \\
\textnormal{(b)}&Assume $P[ W > 0] > 0$, $P[ \xi = \lambda] < 1$ and \textnormal{(IND)}. Then,\\ \\
\multicolumn{2}{c}{$\quad \quad \quad P[a < W < b] > 0 \quad \textnormal{for~all~} 0 \leq a < b \leq \infty $.}\\~\\
\end{tabular}
\end{theorem}

The proofs of Theorem 1 and Theorem 2 are quite similar and therefore together contained in Section 4. The arguments needed for the other two theorems are rather different and slightly more technical. Therefore the proofs of Theorem 3 and Theorem 4 are carried out separately in Section 5 and Section 6.

\section{Relation to the Random Difference Equation}

\noindent In this section we always assume $\xi \equiv \lambda$. Then (1) simplifies into
\[
Z_{n+1} = ( \lambda Z_n - Y_{n+1})_+, \quad n \geq 0.
\]
Consider the process $\big( \hat{Z}_n  \big)_{n \geq 0}$ defined by $\hat{Z}_0 := Z_0 = k$ and
\[
\hat{Z}_{n+1} := \lambda \hat{Z}_n - Y_{n+1}, \quad n \geq 0.
\]
Then, by induction over $n \geq 0$, we find $Z_n = \big(\hat{Z}_n \big)_+$ and
\[
\hat{Z}_n = \lambda^n k -  \sum\limits_{j=1}^{n} Y_j \lambda^{n-j}.
\]

\noindent Hence we obtain for $m \geq 0$
\begin{align}
P[ Z_{n} > m ] &= P[ \hat{Z}_{n} > m] = P \bigg[ k - \sum\limits_{j=1}^n \lambda^{-j} Y_j > m \lambda^{-n} \bigg]\\
& = P \left[ \hat{X}_n < k - m \lambda^{-n} \right],
\end{align}
where $\big( \hat{X}_n \big)_{n \geq 0}$ is the autoregressive process defined by $\hat{X}_0 := 0$ and
\begin{align*}
\hat{X}_{n+1} := \lambda^{-1} \hat{X}_n + Y_{n+1}, \quad n \geq 0.
\end{align*}
The study of this random difference equation was initiated by Kesten in \cite{Kes73} in the more general random-coefficient version
\[ X_{n+1} := A_{n+1} X_n + Y_{n+1}, \quad n \geq 0,   \]
where the sequence $(A_n,Y_n)_{n \geq 1}$ is typically assumed to be i.i.d.\ and independent of $X_0$. In the contractive or subcritical case
\[ E[ \log A_1] < 0   \]
it is well-established, that the condition
\[ E[ \log_+ Y] < \infty  \]
is related the existence of a stationary solution for $(X_n)_{n \geq 0}$, see e.g.\ Theorem 1.6 in \cite{Ver79} or Theorem 2.1.3 in \cite{BDM16}. This can be explained in the following way. For fixed $n \geq 0$ we know by exchangeability
\[
 X_n \overset{\textnormal{d}}{=} \sum\limits_{j=1}^n A_1 \cdots A_{j} Y_{j+1} =: X'_n,
\]
and $X'_n \rightarrow X_\infty$ a.s. for $n \rightarrow \infty$, provided the existence of the limit
\[
X_\infty := \sum\limits_{n \geq 0} A_1 \cdots A_n Y_{n+1}.
\]
The existence of this limit is related to the condition $E[ \log_+ Y] < \infty$. For $A_1 \equiv \lambda^{-1}$ and $Y_1 \geq 0$ we can e.g.\ use Lemma 1 of Section 4 to conclude $X_\infty < \infty$ almost surely if $E[ \log_+ Y_1] < \infty$ and $X_\infty = \infty$ almost surely otherwise. Since inserting $m=0$ into (4) and (5) gives
\begin{align}
P[ \tau = \infty]  = \lim\limits_{n \rightarrow \infty} P[Z_n > 0] = \lim\limits_{n \rightarrow \infty} P \big[ \hat{X}_n < k \big] = P [ \hat{X}_\infty < k],
\end{align}
where
\[
\hat{X}_\infty := \lambda^{-1} \sum\limits_{n \geq 0} \lambda^{-n} Y_{n+1},
\]
we can recover the statement of Theorem 1 in this way. Moreover, consider (6) and the following result obtained by Grincevi\v{c}ius in \cite{Gri75}.

\begin{theorem}[Grincevi\v{c}ius]
Assume that $P[ Y_1 > t]$ is regularly varying for $t \rightarrow \infty$ with index $\alpha \in (0,\infty)$, $E[ A_1^\alpha] < 1$ and $E[ A_1^\beta] < \infty$ for some $0 < \beta < \alpha$. Then,
\[
\lim\limits_{k \rightarrow \infty} P[X_\infty > k]~P[Y_1>k]^{-1} = \sum\limits_{j=0}^\infty E[ A_1^\alpha]^j.
\]
\end{theorem}

In our specific case $\xi \equiv \lambda$ we can apply this theorem with the assumption $A_1 \equiv \lambda^{-1}$ to recover the asymptotic formula obtained for $(q_k)_{k \geq 1}$ as $k \rightarrow \infty$ in Theorem 3. For clarity, let us state that $X_\infty$ and $\hat{X}_\infty$ differ by the constant $\lambda^{-1}$, which explains why the limit in Theorem 3 is $\lambda^{-\alpha}/(1-\lambda^{-\alpha})$ and not $1/(1-\lambda^{-\alpha})$.

It is worth mentioning that Grey questioned some parts of the original proof and gave a new improved version of Theorem 5 in \cite{Grey94}.\\

Finally, observe that in our case all random variables involved in the definition of $(\hat{X}_n)_{n \geq 0}$ are nonnegative and hence Kellerer's theory of recurrence and transience of order-preserving Markov chains is available, see \cite{Kel92} and \cite{Kel06}. By again inserting $m=0$ in (4) and (5) we obtain
\[
E[\tau] = \sum\limits_{n \geq 0} P[ \tau > n] = \sum\limits_{n \geq 0} P[ Z_n > 0] = \sum\limits_{n \geq 0} P \big[ \hat{X}_n < k \big],
\]
and hence conclude that $E[ \tau] = \infty$ if and only if $( \hat{X}_n)_{n \geq 0}$ is recurrent. A recent result by Zerner, see \cite{Zer18}, states, that this rather generally is the case if and only if there exists $b \in (0, \infty)$ with
\[
 \sum\limits_{n \geq 1} \prod\limits_{m=1}^n P[ Y \leq b \lambda^m ] = \infty.
\]
Clearly, this result characterizes the finiteness of $E[\tau]$ exactly and hence more precisely than Theorem 2.\\

Interestingly enough, Zerner's criterion does not only apply to more general random-coefficient autoregressive processes but also to subcritical branching processes with immigration. For this class of branching processes the existence of a stationary solution is again related to the logarithmic moment of the immigration component, see e.g.\ \cite{Qui70} or Theorem A in \cite{Pak79}.

\section{Proofs of Theorem 1 and Theorem 2}

\noindent We start to prepare our proofs with two simple lemmas.

\begin{lemma}
Let $(U_n)_{n \geq 0}$ be a sequence of i.i.d.\ nonnegative random variables. Then
\[
\limsup\limits_{n \rightarrow \infty} \frac{U_n}{n} = \left\{ \begin{array}{l l}
0,&\quad \textnormal{if} \quad E[U_1] < \infty\\ \\
\infty,&\quad \textnormal{if} \quad E[U_1] = \infty.
\end{array} \right.
\]
\end{lemma}~\\

\begin{proof}[Proof]
Recall that $E[U_1] = \infty$ if and only if $\sum_{n \geq 0} P[ U_n \geq c n]$ diverges for all $c \in (0,\infty)$. Hence the claim follows by applying both the first and second part of the Borel-Cantelli lemma.
\end{proof}

The use of Lemma 1 is known in the context of supercritical branching processes with immigration when it is natural to ask when the immigration component accelerates the asymptotic growth. In this context, Lemma 1 allows one to easily obtain some of Seneta's classical results, see \cite{Sen70} and e.g.\ Section 3.1.1 in Dawson's lecture notes \cite{Daw17}.\\

We will also apply the following concentration estimate, which can be seen as a weaker but more general form of Chebyshev's inequality.

\begin{lemma}
Let $(V_n)_{n \geq 0}$ denote a sequence of i.i.d.\ random variables and $S_n := \sum_{j=1}^n V_j$ for all $n \geq 1$. Assume $E[V_1]=0$ and that there exists $\delta \in (0,1]$ with $c:= E[ |V_1|^{1+\delta}] < \infty$. Then,
\begin{align*}
P[ |S_n | > t] \leq  2 c n t^{-1-\delta} \quad \textnormal{for~all~} n \geq 1,~t \in (0,\infty).
\end{align*}
\end{lemma}

\begin{proof}
By a classical result due to von Bahr and Esseen, see \cite{vBE65}, 
\[
E \left[ \vert S_n \vert^{1+\delta}  \right] \leq 2 c n \quad \textnormal{for~all~} n \geq 1.
\]
Therefore the claim follows by applying Markov's inequality.
\end{proof}~

Let us now briefly introduce some notation. The branching process, which is obtained from $(Z_n)_{n \geq 0}$ by neglecting any emigration, will be denoted by $(Z'_n)_{n \geq 0}$. Formally, $Z'_0 := k \geq 1$ and
\[
Z'_{n+1} := \sum\limits_{j=1}^{Z'_n} \xi_{n+1,j}, \quad n \geq 0.
\]
We will also work with the stopping time
\[ \tau' := \inf \{ n \geq 1 ~ \big\vert~ Z'_{n+1} \leq Y_{n+1} \}.   \]
Observe that by definition $Z_n \leq Z'_n$ and $\tau \leq \tau'$ almost surely.

\begin{proof}[Proof of Theorem 1]
(ii)$\Longrightarrow$(i). Choose a $\varepsilon > 0$ and set
\[
T := \inf \left\{ n \geq 1~\vert~ Z'_m \leq ( \lambda + \varepsilon)^m~ \textnormal{for~all~} m \geq n \right\}.
\]
Then, for fixed $n \geq 1$, Markov's inequality gives
\[  P \left[ Z'_n > ( \lambda + \varepsilon)^n \right] \leq  k \left( 1 + \frac{\varepsilon}{\lambda} \right)^{-n}.
\]
Hence the Borel-Cantelli lemma gives $T < \infty$ almost surely. Moreover,  applying Lemma 1 with $U_n := \log_+ Y_n$ gives $Y_n \geq ( \lambda + \varepsilon)^n$ for infinitely many $n \geq 1$ almost surely. This implies $\tau \leq \tau' < \infty$ almost surely.\\

\noindent $\neg$(ii)$\Longrightarrow \neg$(i). By truncating the offspring distribution and working with stochastic dominance, we can assume that the number of children of each individual is a.s. bounded and $\sigma^2 := \textnormal{Var}[\xi] \in [0,\infty)$.

Fix a $\varepsilon > 0$ with $\lambda_1 := \lambda - 2 \varepsilon > 1$ and let $\lambda_0 := \lambda - \varepsilon $. Then, for all $n \geq 1$, consider the following events
\begin{align*}
A_n := \left\{ \sum\limits_{j=1}^{\lfloor \lambda_0^n \rfloor} \xi_{n+1,j} \geq  \left( \lambda - \frac{\varepsilon}{2} \right) \left\lfloor \lambda_0^n \right\rfloor \right\}, \quad \quad \quad B_n := \big\{ Y_n \leq \lambda_1^n \big\}.
\end{align*}
For all $n \geq 1$ we find by Chebyshev's inequality
\begin{align*}
P[ A_n^c] &\leq P \left[ \Bigg\vert \sum\limits_{j=1}^{\lfloor \lambda_0^n \rfloor} \xi_{1,j} - \lambda \lfloor \lambda_0^n \rfloor \Bigg\vert > \frac{\varepsilon}{2} \lfloor \lambda_0^n \rfloor \right] \leq \left( \frac{\varepsilon}{2} \right)^{-2} \frac{\sigma^2}{\lfloor \lambda_0^n \rfloor}.
\end{align*} 
Since $\lambda_0 > 1$ we can apply the Borel-Cantelli lemma to conclude that almost surely only finitely many events $A_n^c$, $n \geq 1$, do occur. On the other hand, by Lemma 1, we know that almost surely all but finitely many events $B_n$, $n \geq 1$, do occur. Also note that $(A_n)_{n \geq 1}$ is a sequence of independent events, and so is $(B_n)_{n \geq 1}$. Clearly $P[A_n] > 0$ for all $n \geq 1$ and there exists $N \in \mathbb{N}$ with $P[B_n] > 0$ for all $n \geq N$. All in all, we therefore can fix a $n_0 \in \mathbb{N}$ such that
\begin{align}
\left( \lambda - \frac{\varepsilon}{2} \right) \left\lfloor \lambda_0^n \right\rfloor - \lambda_1^n &\geq \left\lfloor \lambda_0^{n+1} \right\rfloor \quad \textnormal{for~all~} n \geq n_0,\\
\min \left( P[A],P[B] \right) &> \frac{1}{2}, \quad \textnormal{where}~ A := \bigcap_{n \geq n_0} A_n, \quad B := \bigcap_{n \geq n_0} B_n.
\end{align}
Then, by using (8), we find
\[ P [ A \cap B] = P[ A] + P [ B] - P [A \cup B] \geq P[ A] + P [B] - 1 > 0.  \]
Finally, by recalling hypothesis (H), we may increase the value of the initial state $Z_0 = k$ to ensure that
\[
P[ C] > 0, \quad \quad \quad \textnormal{where}~~C := \left\{ Z_{n_0} \geq \left\lfloor \lambda_0^{n_0} \right\rfloor \right\}.
\]
By inserting our construction of the events $A$ and $B$ and using (7), an inductive argument yields $Z_n \geq \lfloor \lambda_0^n \rfloor$ for all $n \geq n_0$ on the event $A \cap B \cap C$. Since $A \cap B$ and $C$ are independent events by definition,
\[ P [ \tau = \infty] \geq P [ A \cap B \cap C] = P[A \cap B]~P[C] > 0.  \qedhere \]
\end{proof}~

In fact, a careful look at the second part of this proof reveals the following result, which we need for the proof of the Kesten-Stigum theorem.

\begin{proposition}
Assume $E[ \log_+ Y] < \infty$. Then $q_k \rightarrow 0$ for $k \rightarrow \infty$.
\end{proposition}

\noindent The proof of this Proposition is left to the reader.

\begin{proof}[Proof of Theorem 2]
(I). Since $\tau \leq \tau'$ it suffices to verify $E[\tau'] < \infty$. Fix $\varepsilon > 0$ and $r \in (0, \infty)$ according to the assumption and set 
\begin{align*}
T &:= \inf \{ n \geq 1~ \vert~ Z'_m \leq r (\lambda + \varepsilon)^{m-1} \textnormal{~for~all~} m \geq n \}, \\
\hat{T} &:= \inf \left\{ n > T ~\vert~Y_n > r (  \lambda + \varepsilon)^n \right\}.
\end{align*}
Then $\tau' \leq \hat{T}$ almost surely by construction and hence it suffices to prove $E[ \hat{T}] < \infty$. Note that for all $n \geq 1$ Markov's inequality gives
\begin{align}
P[ T= n] \leq P \left[ Z'_{n-1} > r \left( \lambda + \varepsilon \right)^{n-2} \right] \leq \frac{k}{r} \left( 1 + \frac{\varepsilon}{\lambda} \right)^{-n+2}.
\end{align}
Moreover, since $\left( (\xi_{n,j})_{j \geq 1}),Y_n \right)_{n \geq 1}$ is i.i.d.,  we know
\begin{align}
E[ \hat{T}] &= \sum\limits_{n \geq 1} E[ \hat{T}~\vert~ T=n|~ P[T=n] = \sum\limits_{n \geq 1} \left( E[T_n] + n \right) P[ T = n],
\end{align}
where
\[
T_n := \inf \left\{ m \geq 1~|~Y_m > r ( \lambda + \varepsilon)^{n+m} \right\}, \quad n \geq 1.
\]
For all $n \geq 1$ we have
\begin{align*}
\displaystyle E[T_n] &= 1 + \sum_{m \geq 1} P[T_n > m] = 1 + \sum\limits_{m \geq 1} \prod\limits_{l=1}^m P \left[ Y \leq r \left( \lambda + \varepsilon \right)^{n+l} \right]\\ 
&= 1 + \bigg( \sum\limits_{m \geq 1} \prod\limits_{l=1}^m P \big[ Y \leq r \left( \lambda + \varepsilon \right)^l \big] \bigg) \bigg( \prod\limits_{l=1}^n P \big[ Y \leq r  \left( \lambda + \varepsilon \right)^l \big] \bigg)^{-1}.
\end{align*}

Due to our choice of $r \in (0,\infty)$ we conclude that $E[T_n] \in (1,\infty)$ for all $n \geq 1$. Now, by inserting this formula for $E[T_n]$ into equation (10) and then applying inequality (9), we deduce that $E[ \hat{T}] < \infty$. \\

\noindent (II). First, note that by possibly increasing $r \in (0,\infty)$ we can guarantee that there exists $n_0 \geq 1$ satisfying both $r n_0^{-\theta} < 1$ and
\[ P \left[ Y \leq \kappa r \right] > 0, \quad \textnormal{where}~~ \kappa := \prod\limits_{n \geq n_0} \left( 1 - r n^{-\theta} \right) \in (0,1).   \]
Fix $r \in (0,\infty)$, $n_0 \geq 1$ and $\kappa \in (0,1)$ accordingly. By possibly increasing $n_0$ we may further assume
\begin{align}
\sum\limits_{n \geq n_0} \prod\limits_{l=n_0}^n P \left[ Y_1 \leq \kappa r \lambda^l l^{-\theta} \right] = \infty.
\end{align}
Fix $\eta$ with $(1+\delta)^{-1} < \eta < 1$. Then, for all $n \geq n_0$, let 
\begin{align*}
N_n &:= \lambda^n  \left( 1 +  \frac{1}{n} \right) \prod\limits_{l=n_0}^n \left( 1 - r l^{-\theta} \right), \quad f_n := \kappa \lambda^{\eta n} , \quad g_n:= \kappa r n^{- \theta} \lambda^n.
\end{align*}
By possibly increasing $n_0 \in \mathbb{N}$ and recalling $\eta < 1$ we find for all $n \geq n_0$
\begin{align*}
\lambda \lfloor N_n \rfloor - \lceil f_n \rceil &\geq \lambda^{n+1} \left( 1 + \frac{1}{n} \right) \prod\limits_{l=n_0}^n \left( 1 - r l^{- \theta} \right) - \kappa \lambda^{\eta n} n - 2\\
&\geq \lambda^{n+1} \left( 1 + \frac{1}{n} \right) \prod\limits_{l=n_0}^n \left( 1 - r l^{- \theta} \right) - \lambda^{\eta n} n^2 \prod\limits_{l=n_0}^n \left( 1 - r l^{-\theta} \right)\\
&= \left( \lambda^{n+1} \left( 1 + \frac{1}{n} \right) - \lambda^{\eta n}  n^2 \right)  \prod\limits_{l=n_0}^n \left( 1 - r l^{-\theta} \right)\\
&=  \left( \lambda^{n+1} + \frac{\lambda^{n+1}}{n+1}  + \frac{\lambda^{n+1}}{n(n+1)} - \lambda^{\eta n} n^2 \right)  \prod\limits_{l=n_0}^n \left( 1 - r l^{-\theta} \right) \\
&\geq \lambda^{n+1} \left( 1 + \frac{1}{n+1} \right) \prod\limits_{l=n_0}^n  \left( 1 - r l^{-\theta} \right).
\end{align*}

\noindent Also, note that by definition of $\kappa$ we know for all $n \geq n_0$
\[
\lceil g_{n+1} \rceil \leq  g_{n+1} + 1 \leq \lambda^{n+1} \left( 1 + \frac{1}{n+1} \right) \left( \prod\limits_{l=n_0}^n \left( 1 - r l^{-\theta} \right) \right) r (n+1)^{-\theta}.
\]
By combining the previous two estimates, we directly find for all $n \geq n_0$
\begin{align}
\lambda \lfloor N_n \rfloor - \lceil f_n \rceil - \lceil g_{n+1} \rceil \geq N_{n+1}.
\end{align}
For all $n \geq n_0$ we consider the event
\[
D_n := \bigg\{ \sum\limits_{j=1}^{\lfloor N_n \rfloor} \xi_{n,j} \geq \lambda \lfloor N_n \rfloor - \lceil f_n \rceil \bigg\}.
\]
Then, by applying Lemma 2, we know that there exists $c \in (0,\infty)$ with
\begin{align*}
P[ D_n^c] \leq P \Bigg[ \bigg\vert \sum\limits_{j=1}^{\lfloor N_n \rfloor} \xi_{1,j} - \lambda \lfloor N_n \rfloor \bigg\vert > \lceil f_n \rceil \Bigg] \leq \frac{2 c \lfloor N_n \rfloor }{\lfloor f_n \rfloor^{1+\delta}}, \quad \textnormal{for~all~} n \geq n_0.
\end{align*}
Recalling our definition of $N_n$, $f_n$, and $\eta$, and noticing that the events $(D_n)_{n \geq n_0}$ are independent, we obtain
\[
P[ D] > 0, \quad \textnormal{where} \quad D := \bigcap\limits_{n \geq n_0} D_n.
\]
Consider the stopping time
\[
T := \inf \left\{ n > n_0~ \vert~ Y_n > g_n  \right\}.
\]
Then, by (11),
\begin{align*}
E[T] =  \sum\limits_{n \geq 0} P[ T > n] = n_0 +1 + \sum\limits_{n \geq n_0} \prod\limits_{l=n_0 + 1}^n P \left[ Y_1 \leq \kappa r l^{-\theta} \lambda^l \right] = \infty.
\end{align*}
Finally, by using (H), we may assume that the initial state $Z_0 = k \geq 1$ is chosen big enough such that
\[
P[C] > 0, \quad \quad \textnormal{where} \quad C := \{ Z_{n_0} \geq N_{n_0} \}.
\]
Note that by construction $C$ and $D$ are independent events. All in all, by (12) we can deduce that $\tau \geq T$ on the event $ B:= C \cap D$, which occurs with a positive probability. Finally, by (IND),
\[
E[ \tau ] \geq E[ \tau 1_B] \geq E[ T 1_B] = E[ T~ \vert ~ B]~P[B]  = E[ T]~P[ B] = \infty. \qedhere
\]
\end{proof}

\section{Proof of Theorem 3}

\noindent For convenience, we split the proof of Theorem 3 into smaller parts by formulating and separately proving the following two lemmas.

\begin{lemma}
Assume \textnormal{(REG)}. Then,
\[
C := \limsup\limits_{k \rightarrow \infty} ~q_k~P[Y>k]^{-1} \leq \frac{\lambda ^{-\alpha}}{1- \lambda^{-\alpha}}.
\]
\end{lemma}

\begin{lemma}
Assume \textnormal{(REG)} and that all exponential moments of $\xi$ are finite. Moreover, let $N \in \mathbb{Z}_{\geq 2} \cup \{ \infty \}$. Then,
\begin{align*}
\liminf\limits_{k \rightarrow \infty} P [ \tau < N ~ \vert ~ Z_0 = k]~P[Y>k]^{-1} &\geq \sum\limits_{l=1}^{N-1} \lambda^{-\alpha l}.
\end{align*}
\end{lemma}

\begin{proof}[Proof of Lemma 3]
By truncating the distribution of $\xi$ and working with stochastic dominance, we may assume that $\xi$ is almost surely bounded and particularly has finite exponential moments.\\

In the first step we will verify $C < \infty$. For this purpose fix $\varepsilon > 0$ such that $\lambda_0 := \lambda -  2 \varepsilon > 1$ and let $\lambda_1 := \lambda - \varepsilon$. Then, by applying Lemma 6 from Appendix 2, there are $c_1,\ldots,c_{N} \in (0,\infty)$ such that the sequence $(x_n)_{n \geq 0}$ defined by $x_0 := 1$,

\[ x_{n+1} := \left\{ \begin{array}{l l}
\lambda_1 x_n - \lambda_0^{n}, &n \geq N\\ \\
\lambda_1 x_n - c_{n+1},& n \leq N - 1,
\end{array} \right.  \]~

\noindent is strictly positive and satisfies $x_n \geq c^n$ for a $c > 1$ and all $n \geq 1$. Furthermore, for all $k \geq 1$ we consider the events
\[
A_{k,n} := \left\{ \sum\limits_{j=1}^{k x_n} \xi_{n+1,j} \geq \lambda_1 k x_n \right\}, \quad n \geq 0, \quad \quad \quad A_k := \bigcap\limits_{n \geq 0} A_{k,n}.
\]
For all $k \geq 1$ we have
\begin{align}
q_k  = P[\tau < \infty,~A_k~\vert~Z_0 = k]  + P \big[ \tau < \infty,~A_k^c~\vert~ Z_0 =k \big],
\end{align}
as well as
\[
P \big[ \tau < \infty,~A_k^c~\vert~Z_0 = k \big] \leq \sum\limits_{n \geq 0} P[A_{k,n}^c],
\]
and hence, by using the Cram\'er-Chernoff method and our knowledge on $(x_n)_{n \geq 0}$, we find that $P[ \tau < \infty,~A_k^c ~\vert~ Z_0 = k] \rightarrow 0$ for $k \rightarrow \infty$ exponentially fast. Therefore, by applying (REG) we deduce
\[
\lim\limits_{k \rightarrow \infty}~P \big[ \tau < \infty,~A_k^c~\vert~Z_0 = k \big]~P[Y>k]^{-1} = 0,
\]
and by recalling (13) we further conclude
\begin{align*}
C = \limsup\limits_{k \rightarrow \infty} P[ \tau < \infty,~A_k~\vert~ Z_0 = k]~P[Y>k]^{-1} .
\end{align*}
Fix $k \geq 1$ and let $Z_0 = k$. Then, by construction of $A_k$ and $(x_n)_{n \geq 0}$,
\[
\{ \tau < \infty \} \cap A_k \subseteq \bigcup\limits_{n=1}^{N -1} \{ Y_n > k c_{n+1}  \} ~\cup  \bigcup\limits_{n \geq N} \{ Y_n > k \lambda_0^n \},
\]
and therefore
\begin{align}
C \leq \sum\limits_{n=1}^{N-1} P[ Y _n > c_{n+1} k] + P \left[ \sum\limits_{n \geq N} Y_n \lambda_{0}^{-n} > k \right].
\end{align}
Note that, on the one hand due to (REG),
\begin{align*}
\lim\limits_{k \rightarrow \infty} \sum\limits_{n=1}^{N-1} P[Y > c_{n} k]~P[Y>k]^{-1} = \sum\limits_{n=0}^N c_{n+1}^{-\alpha}  < \infty,
\end{align*}
and on the other hand, by Theorem 5,
\[
\limsup\limits_{k \rightarrow \infty} P \bigg[ \sum\limits_{n \geq N} Y_n \lambda_{0}^{-n} > k \bigg]~P[Y>k]^{-1} < \infty.
\]
All in all, by (14) we conclude that $C < \infty$.\\

In the second step we fix $\varepsilon > 0$ with $\lambda - \varepsilon > 1$ and $0 < \delta < \lambda-\varepsilon$. Then, for all $k \geq 1$,
\begin{align*}
q_k  &\leq P[ Y_1 >  k(\lambda - \varepsilon)] + P[\tau < \infty,~Y_1 \leq \delta k~\vert~ Z_0 = k]\\
& \quad + P[\tau < \infty, ~\delta k \leq Y_1 \leq k (\lambda - \varepsilon)~\vert~ Z_0 = k].
\end{align*}
Recall that $(Z_n)_{n \geq 0}$ is a time-homogeneous Markov chain, which is monotone with respect to the initial state. Hence we obtain
\begin{align*}
C &= \limsup\limits_{k \rightarrow \infty} P[ Y>k]^{-1}~q_k \\
&\leq \limsup\limits_{k \rightarrow \infty} P[Y>k]^{-1} P[ Y > k (\lambda - \varepsilon)] + \limsup\limits_{k \rightarrow \infty}  P[Y>k]^{-1} q_{(\lambda - \varepsilon - \delta)k}\\
& \quad  + \limsup\limits_{k \rightarrow \infty} P[Y>k]^{-1} P[ Y_1 \geq \delta k]~ q_{(\varepsilon/2) k},
\end{align*}
where we use the notation $q_r := q_{\lfloor r \rfloor}$ for $r \in \mathbb{R}$. Now, by applying (REG) and $C < \infty$, we conclude for the three summands separately
\begin{align*}
\lim\limits_{k \rightarrow \infty} P[Y>k]^{-1} P[Y>k ( \lambda - \varepsilon)] &= (\lambda - \varepsilon)^{-\alpha},\\
\limsup\limits_{k \rightarrow \infty}~ P[Y>k]^{-1} q_{(\lambda - \varepsilon - \delta)k} &\leq C (\lambda - \varepsilon - \delta)^{-\alpha},\\
\lim\limits_{k \rightarrow \infty} P[Y \geq \delta k]~q_{(\varepsilon/2) k}~P[Y>k]^{-1} &= 0.
\end{align*}

\noindent Hence
\[
C \leq (\lambda - \varepsilon)^{-\alpha} + C ( \lambda - \varepsilon - \delta)^{-\alpha},
\]
and the claim follows by letting $\delta \searrow 0$ and $\varepsilon \searrow 0$.
\end{proof}

\begin{proof}[Proof of Lemma 4]
Due to monotonicity it suffices to prove the claim for $2 \leq N < \infty$. Fix $\varepsilon > 0$.  For all $k \geq 1$ and $l=0,\ldots,N-1$ define
\[
A_{k,l} := \left\{ \sum\limits_{j=1}^{\lceil k (\lambda + \varepsilon)^l \rceil} \xi_{l+1,j} \leq k (\lambda + \varepsilon)^{l+1} \right\}, \quad A_k := \bigcap\limits_{l=0}^{N-1} A_{k,l}.
\]
Then, for all $l=0,\ldots,N-1$, $P[A_{k,l}^c] \rightarrow 0$ for $k \rightarrow \infty$ exponentially fast due to the Cram\'er-Chernoff method. Hence, by (REG),
\begin{align*}
L_N^- := &\liminf\limits_{k \rightarrow \infty} P[ \tau < N~|~Z_0 = k]~P[Y>k]^{-1} \\
= &\liminf\limits_{k \rightarrow \infty} P[ \tau < N,~ A_k~|~ Z_0 = k]~P[Y > k]^{-1},
\end{align*}
and by definition of $A_k$ further
\[ L_N^- \geq \limsup\limits_{k \rightarrow \infty} P \left[ \exists l \in \{ 1,\ldots,N-1 \}: Y_l \geq k ( \lambda + \varepsilon)^l \right] ~P[Y > k]^{-1}. \]
Now, by applying the inclusion-exclusion principle, recalling that the sequence $(Y_m)_{m \geq 1}$ is i.i.d.\ and working with (REG), we obtain
\[
L_N^- \geq \sum\limits_{l=1}^{N-1} \lim\limits_{k \rightarrow \infty} P \left[ Y_1 \geq k ( \lambda + \varepsilon)^l \right] P[Y>k]^{-1} = \sum\limits_{l=1}^{N-1} ( \lambda + \varepsilon)^{-\alpha l}.
\]
The claim now follows by letting $\varepsilon \searrow 0$.
\end{proof}

\begin{proof}[Proof of Theorem 3]
In view of Lemma 3 and Lemma 4, it suffices to prove that for fixed $2 \leq N < \infty$
\[
L_N^+ := \limsup\limits_{k \rightarrow \infty} P[\tau < N~\vert~Z_0=k] \leq \sum\limits_{l=1}^{N-1} \lambda^{-\alpha l}.
\]
By the same arguments as in the proof of Lemma 3 we may assume that all exponential moments of $\xi$ are finite. We will verify the claim by showing that for all $\varepsilon_1 > 0$ satisfying $\lambda_0 := \lambda - 2 \varepsilon_1 > 1$ we have
\begin{align}
L_N^+ \leq \sum\limits_{l=1}^{N-1} \lambda_0^{-\alpha l}.
\end{align}
Let $\lambda_1 := \lambda - \varepsilon_1$. For all $k \geq 1$ and $l=1,\ldots,N$ we define the events
\[
B_{k,l} := \left\{ \sum\limits_{j=1}^{\lfloor k \lambda_1^l \rfloor} \xi_{l+1,j} \geq k \left(\lambda - \frac{\varepsilon_1}{2} \right) \lambda_1^l \right\}, \quad B_k := \bigcap\limits_{l=1}^{N-1} B_{k,l}.
\]
For all $l=1,\ldots,N-1$, the Cram\'er-Chernoff method implies that $P[B_{k,l}^c] \rightarrow 0$ for $k \rightarrow \infty$ exponentially fast, and hence
\[ 
L_N^+ = \limsup\limits_{k \rightarrow \infty} P[ \tau < N,~B_k~\vert~Z_0 = k]~P[Y>k]^{-1}.
\]
Consider the event
\[
C_k := \left\{ \exists l \in \{ 1,\ldots,N-1 \} : Y_l > k \lambda_0^l  \right\}
\]

\noindent Then, by (REG),
\begin{align*}
&\limsup\limits_{k \rightarrow \infty} P[C_k]~P[Y>k]^{-1} \\
\leq &\limsup\limits_{k \rightarrow \infty} \sum\limits_{l=1}^{N-1} P[ Y_1 > k \lambda_0^l ]~P[Y>k]^{-1} =   \sum\limits_{l=1}^{N-1} \left( \lambda - \varepsilon_1 \right)^{-\alpha l},
\end{align*}
and hence, in order to obtain the inequality (15), it suffices to show
\begin{align}
\lim\limits_{k \rightarrow \infty} P[ \tau < N,~B_k,~C_k^c~\vert~ Z_0 = k]~P[Y>k]^{-1} = 0.
\end{align}
Let $\varepsilon_2 \in (0,1)$ and introduce for all $k \geq 1$ the random variables
\begin{align*}
 R_k &:= \# \left\{ l=1,\ldots,N-1~\vert~ Y_l \geq \varepsilon_2 k \lambda_0^l  \right\}, \\
 T_k &:= \inf \left\{ l \geq 1 ~\vert~ Y_l \geq \varepsilon_2 k \lambda_0^l \right\}.
\end{align*}
Then, since (REG) holds and $(Y_m)_{m \geq 1}$ is i.i.d., we easily obtain
\begin{align}
\limsup\limits_{k \rightarrow \infty} P[ \tau < N,~B_k,~C_k^c,~R_k \geq 2~\vert~ Z_0 = k]~P[Y>k]^{-1} = 0.
\end{align}
On the other hand, if $\varepsilon_2$ is chosen small enough, then, by Lemma 6,
\begin{align}
P[ \tau < N,~B_k,~C_k^c,~R_k = 0~\vert~ Z_0 = k] = 0 \quad \textnormal{for~all~} k \geq 1.
\end{align}
Combining (17) and (18), in order to verify (16), we only need to show
\begin{align}
\limsup\limits_{k \rightarrow \infty} P[ \tau < N,~B_k,~C_k^c,~R_k =1~\vert~ Z_0 = k]~P[Y>k]^{-1} = 0.
\end{align}
Note that
\begin{align*}
&P[ \tau < N,~B_k,~C_k^c,~R_k =1 ~\vert~ Z_0 =k]\\
=& \sum\limits_{l=1}^{N-1} P[ \tau < N,~B_k,~C_k^c,~R_k =1,~T_k=l ~\vert~ Z_0 =k],
\end{align*}
If $\varepsilon_2 > 0$ is chosen small enough, inserting the definition of all of our random variables and applying Lemma 7 with $a:= \lambda_1 = \lambda - \varepsilon_1$ gives
\[
P[ \tau < N,~B_k,~C_k^c,~R_k =1,~T_k=l ~\vert~ Z_0 =k] = 0~~ \textnormal{for~}k~\textnormal{large~enough}.
\]
\end{proof}

\section{Proof of the Kesten-Stigum theorem}

Again, we will work with the branching process $(Z'_n)_{n \geq 0}$. However, in some parts of the proof we shall consider the more general case $Z'_0 = k' \geq 1$, when generally $Z_0 = k \neq k'$. The extinction probability of $(Z'_n)_{n \geq 0}$ given $Z'_0 = 1$ will be denoted by $q' \in [0,1)$. We also recall the existence of the almost sure martingale limit
\[
W' := \lim\limits_{n \rightarrow \infty} \lambda^{-n} Z'_n \in [0,\infty).
\]

\begin{lemma}[Decomposition]
Fix $k_0 > k$ with $P[Z_1 = k_0] > 0$ and $k':= k_0 - k$. Let $Z_1^{(1)} := k$, $Z_1^{(2)}:=k'$, and define for $n \geq 1$ recursively
\[
Z_{n+1}^{(1)} := \Bigg( \sum\limits_{j=1}^{Z_n^{(1)}} \xi_{n+1,j} - Y_{n+1} \Bigg)_+, \quad \quad Z_{n+1}^{(2)} := \sum\limits_{j=Z_n^{(1)}+1}^{Z_n^{(1)} + Z_n^{(2)}} \xi_{n+1,j},
\]
Then
\begin{align}
\big( Z_n^{(1)} \big)_{n \geq 1} \overset{\textnormal{d}}{=} (Z_n)_{n \geq 0}, \quad \quad \big( Z_n^{(2)} \big)_{n \geq 1} \overset{\textnormal{d}}{=} (Z'_n)_{n \geq 0},
\end{align}
and if \textnormal{(IND)} holds, then $(Z_n^{(1)})_{n \geq 0}$ and $(Z_n^{(2)})_{n \geq 0}$ are independent. Moreover, for fixed $n \geq 1$,
\begin{align}
Z_n = Z_n^{(1)} + Z_n^{(2)} \quad \textnormal{on~the~event} \quad \{ Z_1 = k_0 \} \cap \{ Z_n^{(1)} > 0 \},\\
Z_n \geq Z_n^{(1)} + Z_n^{(2)} \quad \textnormal{on~the~event} \quad \{ Z_1 \geq k_0 \} \cap \{ Z_n^{(1)} > 0 \}.
\end{align}
\end{lemma}

\begin{proof}[Proof of Lemma 5]
By construction both $(Z_n^{(1)})_{n \geq 0}$ and $(Z_n^{(2)})_{n \geq 0}$ are time-homogeneous Markov chains with the same transition probabilities as $(Z_n)_{n \geq 0}$ respectively $(Z'_n)_{n \geq 0}$. Hence, by comparing the initial states, (20) directly follows.\\

By inserting the definitions of $Z_n^{(1)}$ and $Z_n^{(2)}$, one can straightforward verify both (21) and (22). The details are therefore omitted.\\

\noindent Finally, assume (IND) and let $a_1,a_2,b_1,b_2 \in \mathbb{N}$. Then, for all $n,m \geq 1$,
\begin{align*}
&P \big[ Z_{n+1}^{(1)} = a_2,~Z_{m+1}^{(2)}= b_2 ~\vert~ Z_n^{(1)} = a_1,~Z_m^{(2)} = b_1 \big] \\
=&P \Bigg[ \Bigg( \sum\limits_{j=1}^{a_1} \xi_{n+1,j} - Y_{n+1} \Bigg)_+ = a_2,~ \sum\limits_{j=a_1 + 1}^{a_1 + b_1 } \xi_{m+1,j} = b_2 \Bigg] \\
=& P \Bigg[ \Bigg( \sum\limits_{j=1}^{a_1} \xi_{n+1,j} - Y_{n+1} \Bigg)_+ = a_2 \Bigg]~ P \Bigg[ \sum\limits_{j=a_1 + 1}^{a_1 + b_1} \xi_{m+1,j} = b_2 \Bigg]\\
=& P \left[ Z_{n+1}^{(1)} = a_2~\vert~ Z_n^{(1)} = a_1 \right]~ P \left[ Z_{m+1}^{(2)} = b_2~\vert~Z_m^{(2)} = b_1 \right].
\end{align*}
Hence transitions of $\big( Z_n^{(1)} \big)_{n \geq 1}$ and $\big (Z_n^{(2)} \big)_{n \geq 0}$ are independent and the claim follows by recalling that the initial states are chosen constant.
\end{proof}

\begin{proof}[Proof of Theorem 4]
(a). Let $P[W>0]>0$. Then $P[\tau= \infty]>0$ and hence, by recalling Theorem 1, we immediately obtain $E[ \log_+ Y] > \infty$. Besides, using $Z_n \leq Z'_n$ for $Z_0 = Z'_0 = k$ and applying the classical Kesten-Stigum theorem, see \cite{KS66}, we directly conclude $E[ \xi \log_+ \xi] < \infty$.\\

\noindent On the contrary, let us assume $E[ \log_+ Y] > \infty$ and $E[ \xi \log_+ \xi] < \infty$. Then $P[ \tau = \infty] > 0$ due to Theorem 1 and hence it suffices to show
\[
P[W>0] = P[ \tau = \infty].
\]
In order to obtain this, we note that $\{ W > 0 \} \subseteq \{ \tau = \infty \}$ and verify
\begin{align}
P [ W=0,~ \tau = \infty] = 0.
\end{align}
Due to (H) we know $\{ \tau = \infty \} = \{ Z_n \rightarrow \infty \}$ almost surely. Moreover, we also know that $W$ is monotone with respect to the initial population size $Z_0 = k$. Hence
\begin{align}
P[ W = 0,~\tau = \infty] &\leq \liminf\limits_{ k \rightarrow \infty} P[ W =0~\vert~ Z_0 = k]\\
& = 1 - \limsup\limits_{k \rightarrow \infty} P[ W > 0~\vert~ Z_0 = k].
\end{align}
Observe that for each $k \geq 1$ we may choose $k_0 = k_0(k) > k $ such that
\begin{align}
\limsup_{k \rightarrow \infty} P[ Z_1 \geq k_0(k)~ \vert ~ Z_0 = k] = 1
\end{align}
and $k_0 ( k)- k \rightarrow \infty$ for $k \rightarrow \infty$. Fix $k \geq 1$, $k_0 = k_0 (k)$ and assume $Z_0 = k$. Then, by using the notation introduced in Lemma 5 and (22),
\begin{align}
P[ W > 0] &\geq P [ Z_1 \geq k_0 ]~P \Big[ \forall n: Z_n^{(1)} > 0,~\lim\limits_{n \rightarrow \infty} \lambda^{-n} Z_n^{(2)} > 0 \Big].
\end{align}
Recalling (20) and Proposition 1 we know
\begin{align}
P[ \forall n >0 : Z_n^{(1)} > 0] =  1 - q_k \rightarrow 1 \quad \textnormal{for~} k \rightarrow \infty.
\end{align}
On the other hand, by (20) and the classical Kesten-Stigum theorem,
\[
P \left[ \lim\limits_{n \rightarrow \infty} \lambda^{-n} Z_n^{(2)} > 0 \right] = P \left[ W' > 0~\vert~ Z'_0 = k_0(k) -k \right] = 1 - (q')^{k_0(k) - k},
\]
and, since $k_0 (k) - k \rightarrow \infty$ for $k \rightarrow \infty$, we further obtain
\begin{align}
\lim\limits_{k \rightarrow \infty} P \left[ \lim\limits_{n \rightarrow \infty} \lambda^{-n} Z_n^{(2)} > 0 \right] = 1. 
\end{align}
By combining (26), (28) and (29) with (27), we conclude 
\[
\limsup\limits_{k \rightarrow \infty} P[ W > 0 ~\vert~ Z_0 = k] = 1,
\]
and hence (23) and the claim follows by recalling (24) and (25).\\

\noindent (b). We first prove the claim for $a:=0$ via contradiction. Let us assume that there is $b \in (0,\infty)$ such that $P[ 0 < W < b] = 0$ and $P[ W \geq b + \varepsilon] > 0$ for all $\varepsilon > 0$. Then choose $\varepsilon > 0$ and $\delta > 0$ such that
\begin{align}
\tilde{b} := \lambda^{-1} ( b + \varepsilon) + \delta < b. 
\end{align} 
Also fix a $k_0 > k$ with $P[ Z_1 = k_0] > 0$ and again recall the notation introduced in Lemma 3. Then, by the decomposition (21) and (30),
\begin{align*}
&P[  0 < W <  \tilde{b}] \\
\geq~ &P[ Z_1 = k_0]~ P \bigg[ \lim\limits_{n \rightarrow \infty} \lambda^{-n} Z_n^{(1)} \in \left( 0, \lambda^{-1} (b + \varepsilon) \right),~\lim\limits_{n \rightarrow \infty} \lambda^{-n} Z_n^{(2)} < \delta \bigg].
\end{align*}
By using (IND) and Lemma 5 we further deduce
\begin{align*}
P[0 < W < \tilde{b}] \geq P[ Z_1 = k_0] ~P [ 0 < W < b + \varepsilon]~P \left[0 < W' < \lambda \delta \right],
\end{align*}
where we assume $Z_0 = k$ and $Z'_0 = k_0- k$. Observe that the first two probabilities on the right hand side of this inequality are positive, and this is also true for the third factor. This follows e.g.\ from the fact that $W'$ has a strictly positive Lebesgue density on $(0,\infty)$, compare e.g.\ Chapter 1, Part C of \cite{AN72}.\\

So, all in all, $P[0 < W < \tilde{b}] > 0$, which is a contradiction to our assumptions on $b \in (0,\infty)$. Hence the claim holds for $a:=0$.\\

For arbitrary $a > 0$ we can choose $\varepsilon > 0$ with $\varepsilon < \min (a,b-a)$. Then, by the same arguments as above and assuming $Z_0 = k$, $Z'_0 = k_0 - k$,
\begin{align*}
&P[ a < W < b] \\
\geq ~&P [ Z_1 = k_0]~ P[0 < W < \lambda \varepsilon]~P [\lambda (a - \varepsilon) <  W' < \lambda ( b - \varepsilon)].
\end{align*}
Since we already have verified the claim for $a:=0$ we indeed conclude $P[a < W < b]> 0$.
\end{proof}

\section*{Appendix 1: Ratio tests for Example 1}

\noindent Firstly, let us assume $c > \log \lambda$. Fix $N \in \mathbb{N}$, $\varepsilon > 0$ and let
\[
a_n := \prod\limits_{m=N}^n P \left[ Y \leq  (\lambda + \varepsilon)^m \right], \quad n \geq N.
\]

\noindent By choosing $N$ big enough we can guarantee that for all $n \geq N$
\begin{align*}
\frac{a_{n+1}}{a_n} &= 1 -  P \left[ \log Y >  (n+1) \log (\lambda + \varepsilon) \right] \\
& = 1 - \frac{c}{ \log(\lambda + \varepsilon)} (n+1)^{-1}
\end{align*}

\noindent Fix $\varepsilon > 0$ with $c > \log( \lambda + \varepsilon)$. Then, by Raabe's test, $\sum_{n \geq N} a_n < \infty$.\\

\noindent Secondly, let us assume $c = \log \lambda$. Fix $N \in \mathbb{N}$, $\theta \in (1,\infty)$ and let
\[
b_n := \prod\limits_{m=N}^n P \left[ Y \leq \lambda^m m^{-\theta} \right], \quad n \geq N.
\]

\noindent By choosing $N$ big enough we can guarantee that for all $n \geq N$
\begin{align*}
\frac{b_{n+1}}{b_n} &= 1 -  P \left[ \log Y >  (n+1) \log (\lambda ) - \theta \log(n+1) \right] \\
& = 1 - \frac{\log \lambda}{(n+1) \log(\lambda) - \theta \log(n+1)}\\
&= 1 - \frac{1}{n+1 - \beta  \log(n+1)}, \quad \quad \textnormal{where}~ \beta := \frac{\theta}{\log \lambda}.
\end{align*}
For $n \rightarrow \infty$ we have the series expansion
\[
\frac{1}{n+1 - \beta \log(n+1)} = \frac{1}{n} + \frac{\beta \log(n) - 1}{n^2} + O \left( \frac{1}{n^2} \right),
\]~

\noindent and hence, by applying the Gauss test, we find $\sum_{n \geq N} b_n = \infty$.\\

\section*{Appendix 2: Notes on the recursion $x_{n+1} = a x_n - b_{n+1}$}

The proofs of the following two claims are elementary and therefore are left to the reader.

\begin{lemma}
Let $a \in (1,\infty)$, $x_0 =1 $ and $\varepsilon > 0$. Then there exists $N \in \mathbb{N}$ and $c_1,\ldots,c_N \in (0,\infty)$ such that for the sequence $(x_n)_{n \geq 0}$ defined by
\[ x_{n+1} := \left\{ \begin{array}{l l}
a x_n - (a - \varepsilon)^{n}, &n \geq N\\ \\
a x_n - c_{n+1},& n \leq N - 1,
\end{array} \right.  \]

\noindent there exists $c > 1$ with $x_n \geq c^n > 0$ for all $n \geq 1$. For example, one can fix some $\delta \in (0,1)$, $N \in \mathbb{N}$ and set $c_n := \delta (a- \varepsilon)^n$ for all $n =1,\ldots,N$.
\end{lemma}~

\begin{lemma}
Let $N \in \mathbb{N}$, $a \in (1,\infty)$, $x_0 = 1$ and $\varepsilon_1 > 0$ with $a- \varepsilon_1 > 1$. Then there exists $\varepsilon_2 > 0$ with the following property.\\

\noindent For all $l=1,\ldots,N-1$ the recursion defined by\\
\[ x_{n+1} := a x_n - b_n,  \quad \textnormal{where}~~ b_n := \left\{ \begin{array}{l l}
(a - \varepsilon_1)^l,&n=l,\\ \\
\varepsilon_2 ( a - \varepsilon_1)^n,&n \neq l,\\
\end{array} \right.\\
\]~\\
satisfies $x_{j} \geq \varepsilon_1 > 0$ for all $j=1,\ldots,N$.
\end{lemma}

\end{document}